\documentclass{amsart}

\usepackage{amsmath,amssymb,amsfonts}
\usepackage{eucal}
\usepackage[utf8]{inputenc}

\newcounter{prop}
\newtheorem{proposition}[prop]{Proposition}
\newcounter{coro}
\newtheorem{corollary}[coro]{Corollary}

\newcounter{def}
\newtheorem{definition}[def]{Definition}

\title{Rank of Jacobi operator and existence of quadratic parallel  differential form,  with 
applications to geometry of almost para-contact metric manifolds}
\author{Piotr Dacko}

\begin{document}
\begin{abstract}{It is established that non-isotropic vector field with Jacobi operator  
of maximal rank is an obstacle for the existence of non-trivial second-order symmetric 
parallel tensor field. It follows that such manifold 
as pseudo-Riemannian manifold is locally non-reducible.  In particular result is applied 
 to widely studied classes of almost para-contact metric manifolds - 
 para-contact metric, para-cosymplectic or para-Kenmotsu manifolds  
  satisfying nullity and generalized nullity conditions. As corollary we have the 
  following theorem: almost para-contact metric manifold with 
  maximal rank Jacobi operator of characteristic vector field is locally non-isometric 
  to Riemann product. }
\end{abstract}

\maketitle
\section{Introduction}

Many authors has recently studied  the problem of the existence of non-trivial  parallel 
quadratic form  on almost para-contact metric manifold which  satisfies generalized nullity conditions.
Let us recall it is said that almost para-contact metric manifold
satisfies generalized nullity condition if 
\begin{multline}
\label{nullcond}
R_{XY}\xi = \kappa (\eta(Y)X-\eta(X)Y)+ 
\mu (\eta(Y)hX-\eta(X)Y)+ \\ \nu(\eta(Y)h'X-\eta(X)h'Y),
\end{multline}    
$R$ is the operator of the Riemann curvature, and $\kappa$, $\mu$, $\nu$ are functions,
$\xi$ is the characteristic vector field. The more elaborate terminology is that characteristic vector 
field belongs to generalized nullity distribution.  
For the studied classes of manifolds the results are that in generic case 
  there is no parallel quadratic form, different from metric tensor  up to non-zero  multiplier.  
The proofs of these results are based 
on   properties of Jacobi operator  $X\mapsto J_\xi=R_{X\xi}\xi$,
particularly its algebraic form.
However more careful analysis shows that,  for above mentioned generic  cases, $J_\xi$  
has maximal rank. Going into this direction we have  
\begin{proposition}
There is no non-trivial parallel quadratic form if Jacobi operator 
of characteristic vector field has maximal rank. Manifold as Riemannian
manifold is locally irreducible.
\end{proposition} 
The latter sentence comes from the fact that Riemann product always admit 
non-trivial parallel differential quadratic form. 

 Let's denote by $\chi(\xi,x)$ the  characteristic polynomial of the Jacobi operator
 \begin{equation}
\chi(\xi,x)=x^n-\omega_1(\xi)x^{n-1}+\ldots (\pm 1)^{n-1}\omega_{n-1}(\xi)x,
\end{equation}
we will show the following
\begin{proposition}
Assuming $\xi$ is non-isotropic, $J_\xi$ has maximal rank if and only if the coefficient 
$\omega_{n-1}$ at the lowest power term is non-zero $\omega_{n-1}\neq 0$. 
\end{proposition}
Of course this assumption is always satisfied for characteristic vector field on almost 
para-contact metric manifold. Yet our results are remain valid in wider framework 
of pseudo-Riemannian manifolds. Then fact that $\xi$ is non-isotropic is essential.

Author would like to express his gratitude  to Professor Quanxiang Pan for our fruitfull discussion.

\section{Preliminaries}
Let $(\mathcal M, g)$ be pseudo-Riemannian $n$-dimensional manifold, $\nabla$ denote the Levi-Civita connection, and  $R_{XY}Z$ its curvature  
\begin{equation}
R_{XY}Z = \nabla_X\nabla_Y Z -\nabla_Y\nabla_X Z - \nabla_{[X,Y]} Z,
\end{equation} 
it is assumed that letters $U$, $V$, $X$,... are used  to denote  
vector fields, if it is not stated otherwise. Our convention for the Riemann covariant curvature  tensor  is  
\begin{equation}
R(X,Y,Z,W) = g(R_{XY}Z,W).
\end{equation}
For vector field $\xi$ on $\mathcal M$,  $(1,1)$-tensor field 
$X \mapsto J_\xi X=R_{X \xi}\xi$, is called  Jacobi operator.  
 From the properties of the curvature $J_\xi$ is $g$-self-adjoint and uppper limit of its rank is
$$
   g(J_\xi X,Y)=g(J_\xi Y,X),\quad r < {\rm dim}\,\mathcal M.
$$    
If $g$ is definite, $J_\xi$ is semi-simple: eigenvalues are real and tangent space splits 
into orthogonal direct sum of corresponding eigen-spaces.

Given Riemannian manifold $\mathcal M$, A. Gray \cite{Gray}, considered  so-called k-nullity distribution $\mathcal N_k$, $k=const.$,
distribution where the curvature has algebraic form  as curvature of constant sectional curvature manifold
\begin{equation}
R_{XY}Z = k ( g(Y,Z)X  -  g(X,Z)Y),
\end{equation} 
where $Z$ is section of $\mathcal N_k$.

\subsection{Almost para-contact metric manifolds}(\cite{Blair}, \cite{Zam})

Let $\mathcal M$ be $(2n+1)$-dimensional manifold.
Almost para-contact metric structure $(\varphi, \xi,\eta,g)$  is  quadruple of tensor fields: $(1,1)$-tensor field (affinor) $\varphi$, 
characteristic (or Reeb) vector field $\xi$, characteristic 1-form $\eta$,
and pseudo-Riemannian metric $g$, such that, $\varepsilon=\pm 1$,
\begin{eqnarray}
\label{def1} & \varphi^2 = \varepsilon(Id -\eta\otimes \xi), \quad \eta(\xi) =1, & \\[+4pt]
\label {def2} & g(\varphi X, \varphi Y) = \varepsilon(g(X,Y) -\eta(X)\eta(Y)). &  
\end{eqnarray}

For $\varepsilon=-1$, structure is customary called almost contact metric. 
Eigenvalues of $\varphi$ are imaginary, spectrum contains $\{0,-i,i\}$. It is assumed that the metric 
is strictly positive. The latter condition implies that eigenvalues $-i$, $i$ have the same multiplicity 
$n$. 
The triple $(\varphi, \xi, \eta)$  is called almost contact structure.

For $\varepsilon=+1$, structure is called almost para-contact metric. 
Eigenvalues of $\varphi$ are real, the spectrum is  $\{-1,0,+1\}$.  Tangent space decomposes into direct sum of one-dimensional kernel, 
and  $n$-dimensional eigen-spaces $\mathcal V(\pm 1)$.
From definition it follows that restriction of the metric  to any of eigen-space is null tensor,
and $\mathcal V(\pm 1)$ are maximal in dimension isotropic subspaces.  Signature of $g$ is  
\begin{equation}
 \underbrace{-1,\ldots,-1}_{n}, \underbrace{+1,\ldots,+1}_{n+1} 
\end{equation}
Operators
$P_\pm=\varphi \pm Id$, are orthogonal projectors $P_\pm^2=P_\pm $ , $P_{+}P_{-}=P_{-}P_{+}=0$,
onto eigen-spaces $\mathcal V(\pm 1)$.
We have ${\rm Im}\,P_\pm = \mathcal V(\mp 1)$.
In particular
$g(P_\pm X, P_\pm Y)=0$, $g(P_{\pm}X,P_\mp Y)=g(X,Y)$, for $\eta(X)=\eta(Y)=0$. 
 Per analogy the triple $(\varphi, \xi, \eta)$,  is called almost para-contact structure.

From now on we adopt in this paper the convention where both almost contact 
metric and almost para-contact metric structures are all together called 
almost para-contact metric manifolds. We just follow the line where some authors 
use term pseudo-Riemannian manifold in wider sense: manifold equipped with non-degenerate 
quadratic differential form. So the reader should be aware of this. 

For almost para-contact metric structure tensor field $\varPhi(X,Y) = g(X,\varphi Y)$, is skew-symmetric form called fundamental form.
It satisfies 
\begin{equation}
\eta\wedge \varPhi ^n \neq 0,
\end{equation}
at every point, so $\varPhi$ has maximal rank everywhere, its kernel is spanned by characteristic vector field $\xi$. 

Manifold equipped with almost para-contact metric structure is called almost para-contact metric manifold.
Such manifold is always orientable. 

An important notion is normality. Almost para-contact metric manifold is called normal if 
\begin{equation}
[\varphi,\varphi](X,Y)-2\varepsilon\, d\eta\otimes\xi =0,
\end{equation}
where $[\varphi,\varphi]$ denotes Nijenhuis torsion of $\varphi$ 
\begin{equation*}
[\varphi X,\varphi Y]=\varphi^2 [X,Y]+[\varphi X, \varphi Y]-\varphi ([\varphi X,Y]+[X,\varphi Y]).
\end{equation*}
 
 Non-degenerate hypersurface of almost para-Hermitian manifold can be equipped with 
 almost para-contact metric structure.  Thus such hypersurfaces are one of the fundamental  examples of almost para-contact 
 metric manifolds. 

We just mention some classes of almost para-contact metric manifolds.

\begin{definition}
\label{almcont}
{\rm (\cite{Blair}, \cite{MonNicYud}, \cite{DilPas})}.
{\rm Almost para-contact metric 
manifold $(\mathcal M, \varphi,\xi,\eta, g)$ is called }
\begin{enumerate}
\item para-contact metric 
\begin{equation*}
d\eta = \varPhi,
\end{equation*}
\item almost para-cosymplectic (or almost para-coKaehler)
\begin{equation*}
d\eta =0, \quad d\Phi =0,
\end{equation*}
\item almost para-Kenmotsu 
\begin{equation*}
d\eta =0, \quad d\Phi = 2\eta\wedge\Phi.
\end{equation*}
\end{enumerate}
\end{definition}

Assuming additionally  normality 
we obtain following classes of manifolds: para-Sasakian\footnote{Contact metric and normal, etc}, 
 para-cosymplectic (or para-coKaehler) and para-Kenmotsu.

Let define 
$$
                    h=\frac{1}{2}\mathcal L_\xi \varphi, \quad h'=h\circ\varphi ,
$$ 
$\mathcal L_\xi$ denotes the Lie derivative along $\xi$. Applying $\mathcal L_\xi$ to identity $\varphi \xi=0$,
we obtain $h\xi=0$ ($h'\xi=0$ is evident).
 For given appropriate functions $\kappa$, $\mu$, $\nu$ 
- the choice depends on the structure in question -   almost para-contact metric manifold, such that
\begin{multline}
R_{XY}\xi = \kappa(\eta(Y)X-\eta(X)Y)+\mu(\eta(Y)hX -\eta(X)hY) + \\
                                          \nu(\eta(Y)h'X-\eta(X)h'Y),
\end{multline}
is called $(\kappa,\mu,\nu)$-space or $(\kappa,\mu,\nu)$-almost para-contact metric manifold, 
etc. here authors 
adopt different naming conventions. 
The Jacobi 
operator of characteristic vector field on almost contact metric $(\kappa,\mu,\nu)$-space is 
\begin{equation}
J_\xi: X \mapsto R_{X\xi}\xi = -\kappa\eta(X)\xi + (\kappa Id+\mu h +\nu h')X  ,
\end{equation}
cf.  \cite{BlKoufPap}, \cite{Boeckx}, \cite{DacOl}, \cite{DilPas}. Note that $J_\xi$ has 
maximal rank if and only if 
$$
J_\xi |_{ \eta = 0} = \kappa Id  +\mu h +\nu h',
$$
is invertible on $ \eta = 0$.

\section{Symmetric parallel tensors of pseudo-Riemannian manifolds}
The goal of this section is to prove the following result.
\begin{proposition}
\label{symm}
Let $(\mathcal M,g)$ be $(n+1)$-dimensional pseudo-Riemannian manifold. Let assume there is non-isotropic vector field 
$\xi$, $g(\xi,\xi)\neq 0$, with Jacobi operator   of maximal rank. If $\alpha \neq 0$ is parallel differential quadratic form, 
then it is proportional to metric tensor $\alpha=c g$. 
\end{proposition} 
\begin{proof}
We may assume $g(\xi,\xi)=\epsilon$, $\epsilon =\pm 1$. So  there is $(0,2)$-tensor $\alpha(X,Y)=\alpha(Y,X)$ symmetric and parallel $\nabla \alpha =0$, such that our quadratic differential form is just $X \mapsto \alpha(X,X)$. 
Moreover by assumption $r={\rm dim (Im}\,J_\xi)=n$.
We set 
\begin{equation}
i_\xi\alpha(X)=\alpha(\xi,X), \quad i_\xi g(X)=g(\xi,X). 
\end{equation}
Let consider the case  $i_\xi\alpha\neq 0$. For  
$\alpha$ and $g$  are both parallel 
\begin{equation}
i_\xi\alpha(J_\xi \,\cdot) =0, \quad i_\xi g(J_\xi\, \cdot)=0, 
\end{equation}
hence 
$$
{\rm Im }\,J_\xi \subset ( \iota_\xi\alpha =0 ) \cap ( \iota_\xi g=0 ),
$$
and assumption that $J_\xi$ has maximal rank, implies 
$$
{\rm Im }\,J_\xi = ( \iota_\xi\alpha =0 ) = ( \iota_\xi g=0 ),
$$
in particular forms $i_\xi\alpha$ and $i_\xi g$ are  
collinear 
$i_\xi\alpha = \epsilon\alpha(\xi,\xi)i_\xi g$. 
We apply now the second covariant derivatives $\nabla_{Y,\xi}^2$, and $\nabla_{\xi,Y}^2$ resp. to the both sides of the identity
\begin{equation}
\label{coll}
\alpha(\xi,X)=\epsilon \alpha(\xi,\xi)g(\xi,X),
\end{equation}
we have to take into account that by assumption $\alpha$ is covariant constant. 
Subtracting resulting equations, having in 
mind that $R_{Y\xi}\xi = \nabla_{Y,\xi}^2-\nabla_{\xi,Y}^2$, we obtain 
\begin{equation}
\alpha(J_\xi Y, X) =\epsilon\alpha(\xi,\xi)g( J_\xi Y, X),
\end{equation}
which by maximality  follows
\begin{equation}
\alpha = c g, \quad c=\epsilon\alpha(\xi, \xi) \neq 0,
\end{equation}
and by $0=\nabla \alpha = dc\otimes g $, we have $c=const$.

In the case  $i_\xi\alpha =0$,  in the similar way  as above we find 
$\alpha(X,Y) = 0$, $g(\xi,Y)=0$. For $\xi$ is non-isotropic and
$i_\xi\alpha =0$, $\alpha$ must vanish which contradicts our assumption $\alpha \neq 0$. 
Hence the case $i_\xi\alpha =0$ is not  possible.
\end{proof}

\subsection{Trace form}
Let $\mathcal M$ be $(n+1)$-dimensional pseudo-Riemannian manifold. Let recall formula 
for characteristic polynomial 
\begin{equation}
\chi(\xi,x)=det(x Id -J_{\xi})=x^{n+1} -\omega_1(\xi) x^{n}+\ldots (-1)^{n}\omega_{n}(\xi)x.
\end{equation}
Note that $\omega_i$'s as a functions 
$$
  : \xi \mapsto \omega_{i}(\xi)
$$ 
are all smooth differential forms. Indeed, let 
$J^{i}_{\xi}$ denotes the $i$-th exterior power of $J_{\xi}$.
It is $(i,i)$-tensor field - interpreted as endomorphism acting on $i$-th degree polivectors 
on $\mathcal M$. By definition on simple polivector $\mathcal W = V_1\wedge\ldots V_i$ 
$$
J^{i}_{\xi}(V_1\wedge\ldots V_i) = J_{\xi}(V_1)\wedge\ldots J_{\xi}(V_i),
$$ 
then $\omega_{i}=(-1)^{n} tr(J^{i}_{\xi})$. From the 
definition of Ricci tensor we have $Ric(\xi,\xi)=\omega_{1}(\xi)$. So 
we may think of $\omega_{i}$, $i > 1$, as a kind of Ricci tensors of higher
degrees.   
  
Let us recall the metric tensor gives rise to canonical symmetric bilinear form on 
polivectors: we denote it by $g^{\wedge k}$
\begin{equation}
g^{\wedge k}(X_1\wedge\ldots X_k) = det [ g(X_i,X_j)].
\end{equation}

The metric $g^{\wedge k}$ and $J^{k}_{\xi}$ are compatible in the sense 
that the latter is $g^{\wedge k}$-self-adjoint.
\begin{proposition}
Let assume $J^{n}_{\xi}\neq 0$. Then $J^n_{\xi}$ has rank one. In particular there is $n$-form $\tau_{\xi}$ and 
$n$-multivector $\mathcal W_{\xi}$, such that 
\begin{equation}
J^{n}_{\xi}	= \tau_{\xi}\otimes \mathcal W_{\xi},\quad \omega_{n}(\xi)=\tau_{\xi}(\mathcal W_{\xi}),
\end{equation} 
and $\omega_{n}(\xi) = 0$ if and only if $\xi$ is isotropic.
\end{proposition}
 \begin{proof}
 Let extend $\xi$ to local frame $(\xi, X_1,\ldots,X_{n})$, then 
 \begin{equation}
 \label{base}
\mathcal W_i = \xi\wedge X_{1}\ldots  \wedge \hat{ X}_{i}\wedge\ldots X_n, \quad 1 \leq i \leq n,
\end{equation}  
span the  kernel of $J^{n}_\xi$,  
we see that dimension of kernel is $n$, so 
if $J^{n}_\xi \neq 0$  its image is one-dimensional subspace, so 
there is polivector $\mathcal W$ and $n$-form $\tau$, such that 
\begin{equation}
J^{n}_\xi = \tau\otimes \mathcal W,\quad 
\omega_{n}=tr(J^{n}_\xi) = \tau(\mathcal W),
\end{equation}
$\tau$ and $\mathcal W$ are determined up to re-scaling 
\begin{equation}
	\tau \mapsto f\tau, \quad \mathcal W \mapsto f^{-1}\mathcal W,
\end{equation}
where $f$ is arbitrary non-vanishing function. 

Condition $tr(J^{n}_{\xi})=\tau(\mathcal W)=0$ means that  $\mathcal W$ itself belongs to 
the kernel of $J^{n}_{\xi}$. Every element of the kernel is linear 
combination of polivectors as in (\ref{base}), therefore $\mathcal W$
 has decomposition 
 $$
 \mathcal W = \xi\wedge \mathcal W_{0},
$$   
For some simple polivector $X_{1}\wedge\ldots X_{n}$ we have
$$
 J^{n}_\xi (X_1\wedge\ldots X_{n}) = \mathcal W,
$$
Then 
\begin{equation}
\xi\wedge \mathcal W_0 = \mathcal W = J^{n}_\xi (X_1\wedge\ldots X_{n})
= (R_{X_1\xi}\xi)\wedge\ldots (R_{X_{n}\xi}\xi),
\end{equation}
which follows that up to reoder  
\begin{equation}
R_{X_1\xi}\xi = c_{0} \xi + c_{1}X_{1}+\ldots c_{n}X_{n}, \quad c_0 \neq 0,
\end{equation}
from other hand
$$
  0=g(R_{X_{1}\xi}\xi,\xi)=c_{0}g(\xi,\xi)+\sum\limits_{i=1}^{n}c_{i}g(X_{i},\xi),
$$  
If $g(\xi,\xi)\neq 0$, we 
would take all $X_{1},\ldots X_{n}$, such that $g(X_{i},\xi)=0$ 
then from the above equation we will have $c_{0}=0$ - contradiction. 
Hence $g(\xi,\xi)=0$. 
\end{proof}
Note Jacobi operator has maximal rank if and only if its exterior power $J^n_\xi \neq 0$.
\begin{corollary}
Jacobi operator of non-isotropic vector field has maximal rank if and only if coefficient at lowest term 
of its characteristic polynomial is non-zero. 
\end{corollary}

Note there are non-trivial 
parallel quadratic differential forms on Riemannian products. 
\begin{corollary}
Let $(\mathcal M, g)$ be pseudo-Riemannian manifold. If every point admits
locally defined vector field with maximal rank Jacobi operator, then 
$\mathcal M$ is locally irreducible.
\end{corollary}

\section{Applications to almost para-contact metric manifolds}

Applications are direct. Assume $(\mathcal M, \varphi,\eta,g)$ is almost 
para-contact metric manifold.
\begin{proposition}
If Jacobi operator of characteristic vector field has maximal rank, then 
 non-zero parallel second order differential form is proportional to 
pseudo-length form. In particular manifold is locally irreducible. 
\end{proposition}

In case of $\kappa$-nullity spaces, $\kappa \neq 0$, we just rephrase above result.

\begin{corollary} Non-zero parallel second order differential form on 
$\kappa$-nullity almost para-contact metric manifold is proportional to pseudo-length form
provided $\kappa \neq 0$.
\end{corollary}

The case of $(\kappa,\mu)$-spaces requires study of singular values of 
operator $\kappa Id +\mu h$.  Let denote by 
$\omega(x)=\sum\limits_{i=0}c_ix^{n-i}$, $c_0=1$, the characteristic polynomial of $h$. 
Those values are solutions $(\kappa, \mu)$ of polynomial equation
\begin{equation}
det(\kappa Id+\mu h) = \sum\limits_{i=0}c_i \kappa^{n-i}\mu^i=0,
\end{equation}
are singular values.

The case of almost contact
metric manifold is simpler due to fact that $h$ is  
diagonalizable. Denoting by $\lbrace \lambda_1,\ldots\lambda_k \rbrace$ spectrum of $h$, 
we see that 
$(\kappa,\mu)$ are singular if they satisfy one of the equations
\begin{equation}
  \kappa+\lambda_i \mu =0,\quad i=1,\ldots k.
\end{equation}  

Let have a look at some particular classes of manifolds where we posses more 
detailed information concerning operator $h$. 

\noindent{\bf Example 1.}
Almost Kenmotsu $(\kappa,\mu)$-nullity manifolds. There is strong result which 
asserts that $\kappa = -1$, and $\mu=0$. 
So for such class of manifolds Jacobi operator of $\xi$ is of maximal rank.
If we take instead  
$h'=h\circ \varphi$ (generalized $(\kappa,\mu)'$-nullity spaces), 
then $\kappa \leq -1$ , for $\kappa =-1$, $h'=0$ and for 
$\kappa < -1$, $\mu =- 2$, eigenvalues of $h'$ are $0$, $\pm\sqrt{-k-1}$, 
from these conditions there is one-point singularity $(-2,-2)$. 
Beyond that pair of values vector field $\xi$ has maximal Jacobi operator.

\noindent{\bf Example 2.} Contact metric $(\kappa,\mu)$-nullity spaces.
For such manifold $\kappa \leq 1$, if $\kappa =1$, then $h=0$. For $\kappa <1 $, 
 eigenvalues of $h$ are $0$, $\pm\sqrt{1-k}$ and singular values are 
 pairs $(\kappa,\mu)$ which satisfy $\kappa^2 -(1-k)\mu^2=0$. 
 So, beyond  these points Jacobi operator has maximal rank.
\footnote{Minimal polynomial of $h$ is $x(x^2-(1-\kappa))$} 
 

\thebibliography{99}
\bibitem{Blair}D.~E.~Blair, {\em Riemannian geometry of contact and 
symplectic manifolds}, 
Progress in Math. {\bf 203}, Birkhauser, 2010.
\bibitem{BlKoufPap} D.~E.~Blair, T.~Koufogiorgos, B.~J.~Papantoniou, 
{\em Contact metric manifolds stisfying a nullity condition}, 
Israel J. Math. {\bf 91} (1995), 189--214.
\bibitem{Boeckx} E.~Boeckx, {\em A full classification of almost contact metric 
$(k,\mu)$-spaces}, 
IIllinois J. of Math. {\bf 44 (1)}, (2000), 213--219. 
\bibitem{MonNicYud}B.~Cappelletti-Montano, A.~De Nicola, I.~Yudin, 
{\em Survey on cosymplectic geometry},
Rev. Math. Phys. {\bf 25} (2013), 1343002 (55 pages).
\bibitem{MonErkMur}B.~Cappelletti-Montano, I.~K.~Erken, C.~Murathan, 
{\em Nullity conditions in paracontact geometry}, 
Diff. Geom. Appl. {\bf 30} (2012), 665--693.
 
\bibitem{DacOl} P.~Dacko, Z.~Olszak, 
{\em On almost cosymplectic $(\kappa,\mu,\nu)$-spaces}, 
in: PDEs, Submanifolds and Affine Differential Geometry, Banach Cent. Publ. 
{\bf 69}, 211--220. 

\bibitem{Dac}P.~Dacko, 
{On almost para-cosymplectic manifolds},
Tsukuba J. Math. {\bf 28} (2004), 193--213.

\bibitem{DacErkMur}P.~Dacko, I.~K.~Erken, C.~Murathan, 
{Almost $\alpha$-paracosymplectic manifolds},
J. Geom. Phys. {\bf 88} (2015),  30--51.

\bibitem{DilPas} G.~Dileo, A.~M.~Pastore, {\em Almost Kenmotsu manifolds and nullity 
distributions}, J. Geom. {\bf 93} (2009), 46--61.

\bibitem{Dileo}G.~Dileo, {On the geometry of almost contact metric manifolds 
of Kenmotsu type},
Diff. Geom. Appl. {\bf 29} (2011), 558--564.

\bibitem{Eisen}L.~P.~Eisenhart, 
{Symmetic tensors of the second order whose first covariant derivatives 
are zero},
Trans. Amer. Math. Soc. {\bf 25} (1923)(2), 297--306.

\bibitem{Erd}S.~Erdem, 
{\em On almost (para)contact (hyperbolic) manifolds and harmonicity 
of $(\varphi,\varphi')$-holomoprhic maps between them},
Houston J. Math. {bf 28} (2002), 21--45.

\bibitem{Gray}A.~Gray {\em Spaces of constancy of curavture operators}, 
Proc Amer. Math. Soc. {\bf 17} (1966), 897--902.
\bibitem{PastSalt} A.~M.~Pastore, V.~Saltarelli, 
{\em Generalized nullity distributions on almost Kenmotsu manifolds},
Int. Elect. J. Geom, {\bf 4 (2)} (2011), 168--183.
\bibitem{Salt}V.~Saltarelli, 
{\em Three-dimensional almost Kenmotsu manifolds satisfying certain nullity
conditions},
Bull. Malays. Math. Sci. Soc. (2015) 38:437--459. 
 {\bf 30} (2012), 665--693.
 \bibitem{Zam}S.~Zamkovoy, 
 {Canonical connections on paracontact manifolds},
 Ann. Glob. Anal. Geom. {\bf 36} (2009), 37--60.
\end{document}